\documentclass[a4paper]{amsart}

\usepackage{amsmath,amssymb,amsthm,amsfonts}
\usepackage{hyperref}

\def\epsilon{\varepsilon}
\def\R{\mathbb{ R}}
\def\Q{\mathbb{ Q}}
\def\N{\mathbb{ N}}
\def\<{\langle}
\def\>{\rangle}
\def\norm{\|\hspace*{.35em} \|}

\newtheorem{thm}{Theorem}[section]
\newtheorem{theo}[thm]{Theorem}
\newtheorem{prop}[thm]{Proposition}
\newtheorem{coro}[thm]{Corollary}
\newtheorem{lemma}[thm]{Lemma}

\theoremstyle{definition}
\newtheorem{defi}[thm]{Definition}
\theoremstyle{remark}
\newtheorem{rema}[thm]{Remark}

\begin{document}

\title[Convex compact sets that admit a lsc strictly convex function]{Convex compact sets that admit a lower semicontinuous strictly convex function}

\author{L. Garc\'ia-Lirola}
\address{Departamento de Matem\'aticas, Facultad de Matem\'aticas, Universidad de Murcia, 
30100 Espinardo (Murcia), Spain}
\email{luiscarlos.garcia@um.es}
\author{J. Orihuela}
\address{Departamento de Matem\'aticas, Facultad de Matem\'aticas, Universidad de Murcia, 
30100 Espinardo (Murcia), Spain}
\email{joseori@um.es}
\author{M. Raja}
\address{Departamento de Matem\'aticas, Facultad de Matem\'aticas, Universidad de Murcia, 
30100 Espinardo (Murcia), Spain}
\email{matias@um.es}

\thanks{Partially supported by the grants MINECO/FEDER MTM2014-57838-C2-1-P and Fundaci\'on S\'eneca CARM
19368/PI/14}
\date{October, 2015}



\begin{abstract} We study the class of compact convex subsets of a topological vector space which admits a strictly convex and lower semicontinuous function. We prove that such a compact set is embeddable in a strictly convex dual Banach space endowed with its weak$^*$ topology. In addition, we find exposed points where a strictly convex lower semicontinuous function is continuous. 
\end{abstract}

\maketitle

\section{Introduction}

A well-known result of Herv\'e~\cite{herve} says that a compact convex subset $K \subset X$ of a locally convex space is metrizable if and only if there exists $f\colon K \rightarrow {\mathbb R}$ which is both continuous and strictly convex. It happens that lower semicontinuity is a very natural hypothesis for a convex function, so it is natural to wonder if the existence of a strictly convex lower semicontinuous function on compact convex subset $K \subset X$ of a locally convex space enforces special topological properties on $K$\@. Ribarska proved~\cite{riba, riba2} that such a compact is {\it fragmentable} by a finer metric, and in particular it contains a completely metrizable dense subset. The third named author proved~\cite{raja} that the same is true for the set of its extreme points $\mbox{ext}(K)$\@. On the other hand, Talagrand's argument in~\cite[Theorem 5.2.(ii)]{DGZ} shows that $[0,\omega_1]$ is not embeddable in such a compact set. In addition, Godefroy and Li showed~\cite{GL} that if the set of probabilities on a compact group $K$ admits a strictly convex lower semicontinuous function then $K$ is metrizable.

Our purpose here is to continue with the study of the class of compact convex subsets which admits a strictly convex lower semicontinuous function. We shall denote this class by $\mathcal{SC}$\@. The first remarkable fact that we have got is a Banach representation result.

\begin{theo}\label{main1}
Let $X$ be a locally convex topological vector space and let $K \subset X$ 
be convex compact subset.
Then there exists a function $f\colon K \rightarrow {\mathbb R}$ which is both lower 
semicontinuous and strictly convex if and only if $K$ imbeds linearly into a 
strictly convex dual Banach space $Z$ endowed with its weak$^*$ topology.
\end{theo}

Notice that the strictly convex norm of the dual Banach space in the statement 
is weak* lower semicontinuous, which is a stronger condition that just being a 
strictly convex Banach space isomorphic to a dual space.

If $f\colon K \rightarrow {\mathbb R}$ is a strictly convex function, then the {\it symmetric} 
defined by 
\[ \rho(x,y)=\frac{f(x)+f(y)}{2} - f(\frac{x+y}{2}) \]
provides a consistent way to measure {\it diameters} of subsets of $K$\@. This idea was successfully applied in renorming theory~\cite{MOTV}. We will prove that every nonempty subset of $K$ has slices of arbitrarily small $\rho$-diameter, we can mimic some arguments of the geometric study of the Radon\textendash{}Nikod\'{y}m property which leads to results as the following one.

\begin{theo}\label{main2}
Let $X$ be a locally convex topological vector space and
let $f\colon X \rightarrow {\mathbb R}$ be lower semicontinuous, strictly convex and bounded on compact sets. 
Then for every $K \subset X$ compact and convex, the set of points in $K$ which are both exposed and continuity points of $f|_K$ is dense in $\mbox{ext}(K)$\@.
\end{theo}

The organization of the paper is as follows. In the second section we present stability properties of the class $\mathcal{SC}$, which allow us to prove the embedding Theorem~\ref{main1}. The third and fourth sections are devoted to the search of faces and exposed points of continuity, respectively. Finally, a characterization of the class $\mathcal{SC}$ in terms of the existence of a symmetric with countable dentability index is given in Section 5.

\section{Embedding into a dual space}

Along this section $X$ will denote a locally convex topological vector space. Our first goal is to study the properties of following class of compact sets. 

\begin{defi} The class $\mathcal{SC}(X)$ consists of all the nonempty
compact convex subsets $K$ of $X$ such that there exists a function $f\colon K\rightarrow {\mathbb R}$ which is lower semicontinuous and strictly convex. 
In addition, $\mathcal{SC}$ denotes the class composed of all the families $\mathcal{SC}(X)$ for any locally convex space $X$\@.
\end{defi}

Since a lower semicontinuous function on a compact space attains its minimum, the function $f$ is bounded below.
Later we shall show that we may always take $f$ to be bounded. Notice that metrizable convex compacts admits
continuous strictly convex functions, so they are in the class. In particular, if $X$ is metrizable then
$\mathcal{SC}(X)$ contains all the convex compact subsets of $X$\@. If $X$ is a Banach space endowed with its weak
topology, then $\mathcal{SC}(X)$ is made up of all convex weakly compact subsets as a consequence of the strictly
convex renorming results for WCG spaces.

\begin{prop}\label{property} The class $\mathcal{SC}$ satisfies the following stability properties:
\begin{itemize}
\item[a)] $\mathcal{SC}(X)$ is stable by translations and homothetics;
\item[b)] $\mathcal{SC}$ is stable by Cartesian products;
\item[c)] $\mathcal{SC}$ is stable by linear continuous images;
\item[d)] If $A,B \in \mathcal{SC}(X)$, then $A+B \in \mathcal{SC}(X)$\@.
\end{itemize}
\end{prop}

\begin{proof}
 Statement $a)$ is obvious. To prove $b)$ suppose that $f_i$ witnesses $A_i \in \mathcal{SC}(X_i)$
for $i=1,\dots,n$\@. Then $\sum_{i=1}^{n} f_i \circ \pi_i$, where $\pi_i\colon\bigotimes_{i=1}^{n} X_i \rightarrow X_i$
is the coordinate projection, witnesses that $A_1 \times \dots \times A_n \in \mathcal{SC}(\bigotimes_{i=1}^{n} X_i)$\@.

To prove $c)$ assume that $A \in \mathcal{SC}(X)$ and $T\colon X \rightarrow Y$ is linear and continuous.
Obviously $T(A)$ is convex and compact. Let $f\colon A \rightarrow {\mathbb R}$ be lower semicontinuous and strictly convex.
It is straightforward to check that the function $g\colon T(A) \rightarrow {\mathbb R}$ defined by
\[ g(y) =\inf\left\{ f(x): x \in T^{-1}(y) \right\} \]
does the work. Finally, $d)$ follows by a combination of $b)$ and $c)$\@.
\end{proof}

We will need a kind of external convex sum of convex compact sets.

\begin{defi} Given $A,B \subset X$ convex compact define a subset of $X \times X \times {\mathbb R}$ by 
\[  A \oplus B = \left\{ (\lambda x, (1-\lambda) y, \lambda ): x \in A, y \in B, \lambda \in [0,1] \right\}\,.\]
\end{defi}

\begin{lemma}\label{tech}
Let $A,B \subset X$ be convex compact subsets. Then
\begin{itemize}
\item[a)] $A \oplus B$ is a convex compact subset of $X \times X \times {\mathbb R}$;
\item[b)] if $f\colon A \rightarrow {\mathbb R}$ and $g\colon B \rightarrow {\mathbb R}$ are convex, then 
$h\colon A \oplus B \rightarrow {\mathbb R}$  defined by 
\[ h\bigl((\lambda x, (1-\lambda) y, \lambda)\bigr) = \lambda f(x) + (1-\lambda) g(y) \]
is convex as well;
\item[c)] if  $A,B \in \mathcal{SC}(X)$, then  $A \oplus B \in 
\mathcal{SC}(X \times X \times {\mathbb R})$\@.
\end{itemize}
\end{lemma}

\begin{proof} Compactness is clear in statement $a)$\@. Given 
$(\lambda_i x_i, (1-\lambda_i) y_i, \lambda _i) \in A \oplus B$ for $i=1,2$, just observe that
\begin{align*}
\lefteqn{\left(\frac{\lambda_1 x_1+\lambda_2 x_2}{2}, \frac{(1-\lambda_1) y_1+(1-\lambda_2) y_2}{2}, \frac{\lambda_1+\lambda_2}{2} \right)} \\
 &= \left(\frac{\lambda_1+\lambda_2}{2} \, \frac{\lambda_1 x_1 + \lambda_2 x_2}{\lambda_1 + \lambda_2}, \bigl(1-\frac{\lambda_1+\lambda_2}{2}\bigr) \frac{(1-\lambda_1) y_1 +(1- \lambda_2) y_2}{(1-\lambda_1) + (1 - \lambda_2)}, \frac{\lambda_1+\lambda_2}{2} \right)
\end{align*}
(the case where $\lambda_1=\lambda_2=0,1$ can be handed in a different way). Thus, $A \oplus B$ is convex. For the convexity of function $h$ notice that
\begin{align*}
 \lefteqn{h\bigl((\frac{\lambda_1 x_1+\lambda_2 x_2}{2}, \frac{(1-\lambda_1) y_1+(1-\lambda_2) y_2}{2}, \frac{\lambda_1+\lambda_2}{2}) \bigr)}  \\
 &= \frac{\lambda_1+\lambda_2}{2} f\bigl(\frac{\lambda_1 x_1 + \lambda_2 x_2}{\lambda_1 + \lambda_2}\bigr) +  \left(1-\frac{\lambda_1+\lambda_2}{2}\right) \, g\bigl(\frac{(1-\lambda_1) y_1 +(1- \lambda_2) y_2}{(1-\lambda_1) + (1 - \lambda_2)} \bigr) \\
 &\leq  \frac{\lambda_1+\lambda_2}{2} \, \frac{\lambda_1 f(x_1) + \lambda_2 f(x_2)}{\lambda_1 + \lambda_2} +  \left(1-\frac{\lambda_1+\lambda_2}{2}\right) \, \frac{(1-\lambda_1) g(y_1) +(1- \lambda_2) g(y_2)}{(1-\lambda_1) + (1 - \lambda_2)}  \\
 &= \frac{1}{2} \left(h\bigl((\lambda_1 x_1, (1-\lambda_1) y_1, \lambda _1)\bigr) + h\bigl((\lambda_2 x_2, (1-\lambda_2) y_2, \lambda _2)\bigr)\right)\,.
\end{align*}
If $f$ and $g$ were strictly convex, the above inequality for $h$ would become strict if $x_1 \not = x_2$ or $y_1 \not  = y_2$\@. To overcome this difficulty consider the function 
\[ k\bigl((\lambda x, (1-\lambda) y, \lambda)\bigr) =
h\bigl((\lambda x, (1-\lambda) y, \lambda)\bigr) + \lambda^2 \]
and notice that $\lambda^2$ provides the strict inequality when $x_1=x_2$ and $y_1=y_2$\@.
\end{proof}

\begin{prop}
Suppose that $A,B \in \mathcal{SC}(X)$\@. Then $\mbox{conv}(A \cup B) \in \mathcal{SC}(X)$ and
$\mbox{aconv}(A) \in \mathcal{SC}(X)$\@.
\end{prop}
\begin{proof} Consider the map $T\colon X \times X \times {\mathbb R} \rightarrow X$ defined by
$  T\bigl((x,y,t)\bigr) = x+y $ and  observe that
$T(A \oplus B)= \mbox{conv}(A \cup B)$\@. Since $T$ is linear and continuous, the combination of the previous results gives us that $\mbox{conv}(A \cup B)  \in \mathcal{SC}(X)$\@. The application to the symmetric convex hull follows by applying it with $B=-A$\@.
\end{proof}

\begin{lemma}\label{imbed}
Let $B \subset X$ be a symmetric compact convex set and let $Z = \mbox{span}(B)$\@. Then the following hold:
\begin{itemize}
\item[a)] $Z$, with the norm given by the Minkowski functional of $B$, is isometric to a dual Banach space;
\item[b)] $B$ imbeds linearly into $(Z,w^*)$\@;
\item[c)] if $f\colon X \rightarrow {\mathbb R}$ is convex and lower semicontinuous, then $f|_Z$ is weak* lower semicontinuous.
\end{itemize}
\end{lemma}
\begin{proof} Notice that $Z=\bigcup_{n=1}^{\infty} n B$, and thus the Minkowski functional of $B$ is a norm on 
$Z$\@. Of course, $B$ is the unit ball of $Z$ endowed with this norm. By a result of Dixmier-Ng, see for instance~\cite{NG}, the space $Z$ is isometric to the dual of the Banach space $W$ of all linear functionals $f$ on $Z$ such that $f|_B$ is $\tau$-continuous. 
If $f\colon X \rightarrow {\mathbb R}$ is convex and lower semicontinuous, then the sets $\{f \leq a\}$ are convex
and closed for any $a \in {\mathbb R}$\@. We have $\{f|_Z \leq a\}= \{f \leq a\} \cap Z$, and thus 
$\{f|_Z \leq a\} \cap nB= \{f \leq a\} \cap nB$ is compact, and so it is weak* compact as subset of $Z$
for every $n \in {\mathbb N}$\@.
By the Banach-Dieudonn\'e theorem, $\{f|_Z \leq a\}$ is a weak* closed subset of $Z$\@.
\end{proof}

\begin{proof}[Proof of Theorem~\ref{main1}]
Let $B = \mbox{aconv}(K)$ which is in $\mathcal{SC}(X)$\@. 
The function $f$ witnessing that
$B \in \mathcal{SC}(X)$ is weak* lower semicontinuous and strictly convex.
By Lemma~\ref{imbed} we only need to renorm the dual space $Z$\@. 
Notice that the function $f$ can be taken symmetric and bounded. 
Indeed, for the symmetry just take $g(x)=f(x)+f(-x)$\@.
Now apply the Baire theorem to the $B = \bigcup_{n=1}^{\infty} g^{-1}((-\infty,n])$ 
to obtain a set of the form $\lambda B$ with $\lambda>0$ where $g$ is bounded. 
Then redefine $f$ as  $f(x)=g(\lambda x)$\@.

Without loss of generality we may assume that $f$ takes values in $[0,1]$\@. Consider the function defined on $B_Z$ by
\[ h(x)=\frac{1}{2}( 3\|x\| + f(x) ) \]
and consider the set $C=\{x \in B_Z: h(x) \leq 1 \}$\@. 
Clearly $\frac{1}{3}B_Z \subset C \subset \frac{2}{3}B_Z$, and $C$ is convex, symmetric and weak* closed. Moreover, if $h(x)=h(y)=1$, then $h\bigl(\frac{x+y}{2}\bigr)<1$\@.
Therefore, $C$ is the unit ball of an equivalent strictly convex dual norm on $Z$\@.
\end{proof}

\begin{coro} If $K\in\mathcal{SC}(X)$, then it is witnessed by the square of a lower semicontinuous strictly convex norm defined on $\mbox{span}(K)$\@.
\end{coro}

We shall finish this section by showing the connection between the class $\mathcal{SC}$ and $(\ast)$ property. The following notion was introduced in~\cite{OST} in order to characterize dual Banach spaces that admit a dual strictly convex norm:

\begin{defi} A compact space $K$ is said to have $(\ast)$ if there exists a sequence $(\mathcal{U}_n)_{n=1}^\infty$ of families of open subsets of $K$ such that, given any $x,y\in K$, there exists $n\in\N$ such that:
	\begin{itemize}
	\item[a)] $\{x,y\}\cap\bigcup \mathcal{U}_n$ is non-empty;
	\item[b)] $\{x,y\}\cap U$ is at most a singleton for every $U\in\mathcal{U}_n$\@.
	\end{itemize}
\end{defi}
	Here we are using the agreement that $\bigcup \mathcal{U}_n=\bigcup \{U: U \in \mathcal{U}_n \}$\@. Recall that if $K$ is a subset of a locally convex topological vector space then a slice of $K$ is an intersection of $K$ with an open halfspace. If the elements of $\bigcup_{n=1}^\infty\mathcal{U}_n$ can be taken to be slices of $K$, then $K$ is said to have $(\ast)$ \emph{with slices}. It is shown in~\cite[Theorem 2.7]{OST} that if $Z$ is a dual Banach space then $(B_{Z},w^*)$ has $(\ast)$ with slices if and only if $Z$ admits a dual strictly convex norm. 

\begin{coro}\label{SCstar} Let $(X,\tau)$ be locally convex topological vector space and $K\subset X$ be compact and convex. Then $K\in \mathcal{SC}(X)$ if and only if $K$ has $(\ast)$ with slices. 
\end{coro} 

\begin{proof} By Lemma~\ref{imbed} we may assume that $K \subset Z=\mbox{span}(K)$ has $(\ast)$ with weak* slices. It follows from~\cite[Proposition 2.2]{OST} that then there is a lower semicontinuous strictly convex function defined on $K$\@. On the other hand, assume that $\phi$ witnesses $K\in \mathcal{SC}(X)$\@. For $f\in(X,\tau)^\ast$ and $r\in\R$, denote $S(f,r)= \{x\in K : f(x)>r\}$\@. Consider the families $\{\mathcal{U}_{qr}\}_{q,r\in\Q}$ of open subsets given by
\[ \mathcal{U}_{qr} =  \left\{S(f,r) : f\in (X,\tau)^\ast, S(f,r)\cap\{x:\phi(x)\leq q\}=\emptyset\right\}\,. \]
Let $x\neq y$ be in $K$\@. We may assume that $\phi(x)\leq \phi(y)$\@. Since $\phi$ is strictly convex, there exists $q\in \Q$ such that $\phi(\frac{x+y}{2})<q<\phi(y)$\@. By the Hahn\textendash{}Banach theorem, there is $f\in(X,\tau)^\ast$ and $r\in \Q$ such that $\sup\{f(z): \phi(z)\leq q\} < r < f(y)$\@. Therefore, $S(f,r)\cap\{z:\phi(z)\leq q\}=\emptyset$ and $\{x,y\}\cap \bigcup \mathcal{U}_{qr} \neq \emptyset$\@.

Suppose that $x,y\in S(g,r) \in \mathcal{U}_{qr}$\@. Then $g(x),g(y)>r$ implies $g\bigl(\frac{x+y}{2}\bigr)>r$\@. Hence $\frac{x+y}{2}\notin\{z:\phi(z)\leq q\}$, a contradiction. So $\{x,y\}\cap S(g,q)$ is at most a singleton for each $S(g,q)\in \mathcal{U}_{qr}$\@. 
\end{proof}

\section{Faces of continuity}

We will assume along the section that $Z=W^*$ is a dual Banach space endowed with the weak$^*$ topology.
Therefore any unspecified topological concept (compact, open, \dots) is always referred to the weak$^*$ topology.
The elements of $W$ will be considered as functionals on $Z$\@. Other topological ingredient that we will use
is a symmetric $\rho\colon Z \times Z \rightarrow [0,+\infty)$\@. Recall that a symmetric satisfies $\rho(x,y)=\rho(y,x)$
and $\rho(x,y)=0$ if and only if $x=y$\@. Since a symmetric does not satisfy the triangle inequality, its associated topology is complicated to handle. Nevertheless we have a natural notion of diameter associated to $\rho$ defined by
\[ \rho \mbox{-diam}(A)= \sup\{ \rho(x,y): x,y \in A \} \]

Let us recall the definition of face of a convex set.

\begin{defi}
Let $C \subset Z$ be closed and convex. We say that a closed subset $F \subset C$ is a face if 
there is a continuous affine function $w\colon C \rightarrow {\mathbb R}$ such that 
\[F=\{x \in C: w(x)=\sup\{w,C\}\}\,.\]
In that case we say that the face is produced by $w$\@. In addition, we say that a point $x\in C$ is a exposed point of $C$ if $\{x\}$ is a face of $C$\@. 
\end{defi}

Sometimes the face is produced by an element of the dual. Nevertheless, there may exist continuous affine functions on $C$ that are not the restriction of an element of the dual. 

We shall need the following lemma.

\begin{lemma}[Lemma 3.3.3 of~\cite{bourgin}]\label{bour}
Suppose that $w \in W$ and $\|w\|=1$\@. For $r>0$ denote by $V_r$ the set $rB_Z \cap w^{-1}(0)$\@. Assume
that $x_0$ and $y$ are points of $Z$ such that $w(x_0)>w(y)$ and $\|x_0-y\| \leq r/2$\@. If $u \in W$ satisfies
that $\|u\|=1$ and $u(x_0) > \sup\{u,y+V_r\}$, then $\|w-u\| \leq \frac{2}{r} \|x_0-y\|$\@.
\end{lemma}

First we shall discuss the dual Banach case.

\begin{prop}\label{facedual}
Let $f\colon Z \rightarrow {\mathbb R}$ be a convex lower semicontinuous function which is bounded on compact
subsets. If $K\subset Z$ is compact convex, then there exists a $G_\delta$ dense set of elements of $W$ producing faces where $f|_K$ is constant and continuous.
\end{prop}
\begin{proof}
Define the pseudo-symmetric $\rho$ by the formula 
\[ \rho(x,y)=\frac{f(x)^2+f(y)^2}{2} - f\bigl(\frac{x+y}{2}\bigr)^2\,.\]
We claim that $\rho(x,y)=0$ implies $f(x)=f(y)=f(\frac{x+y}{2})$ (in particular, if $f$ were strictly convex, 
$\rho$ would be a symmetric). Indeed, it follows easily from this observation
\[ \rho(x,y) \geq  \frac{f(x)^2+f(y)^2}{2} - \left(\frac{f(x)+f(y)}{2} \right)^2 =  
\left(\frac{f(x)-f(y)}{2} \right)^2  \geq 0\,.\]
Now we claim that the set $G(K,\epsilon)$ is open and dense in $W$ for $K \subset Z$ compact  convex and $\epsilon>0$, where
\[ G(K,\epsilon) = \left\{ w \in W: \exists a < \sup\{w,K\}, \rho\mbox{-diam}(K \cap \{w>a\}) < \epsilon\right\}\,. \]
Suppose that $w \in G(K,\epsilon)$\@. If $w' \in W$ is close enough to $w$ to fulfill that 
\[\sup\{w',K\} > \sup\left\{w', K \cap \{w \leq a \}\right\}\]
then $w' \in G(K,\epsilon)$ as well. Thus $G(K,\epsilon)$ is open. In order to see that it is also dense, 
fix $w \in W$ and $\delta\leq1/4$\@. Take $x \in K$ and $y \in Z$
with $w(x) >a> w(y)$ for some $a \in {\mathbb R}$\@. Take $r=\sup\left\{\|x'-y\|,x'\in K\right\}/2\delta$, consider the set $V_r$ given by Lemma~\ref{bour} and define the set $C=\mbox{conv}(K \cup (y +V_r))$\@. 
By~\cite[Theorem 1.1]{raja}, the halfspace $\{w>a\}$ contains a point $x_0 \in \mbox{ext}(C)$ where $f|_C$ is continuous. Notice that
$x_0 \in \mbox{ext}(K)$ and $\|x_0-y\|\leq r/2$\@.
There exists $u \in W$ and 
$b \in {\mathbb R}$ such that $u(x_0)>b$, $C\cap\{u>b\}\subset C\cap\{w>a\}$ and $\rho\mbox{-diam}(C \cap \{u>b\} ) <\epsilon$\@. 
In particular $\rho\mbox{-diam}(K \cap \{u>b\} ) <\epsilon$\@. 
Since $C\cap\{u>b\}$ does not meet $y+V_r$, we have $u(x_0) > \sup\{u,y+V_r\}$\@. Thus, $\|w-u\| \leq \frac{2}{r} \|x_0-y\|\leq \delta$\@. That completes
the proof of the density of $G(K,\epsilon)$ in $W$\@.

By the Baire theorem, the set $G(K)=\bigcap_{n=1}^{\infty} G(K,1/n)$ is dense. If $w \in G(K)$
and $s = \sup\{w, K\}$ then
\[ \lim_{t \rightarrow s^{-}} \rho\mbox{-diam}(K \cap \{w > t\} ) = 0\,. \]
In particular, the face $F = K \cap \{w=s\}$ satisfies that $\rho\mbox{-diam}(F)=0$\@. 
That implies that $f$ is constant on $F$\@. Moreover, we claim that any point
$x \in F$ is a point of continuity of $f|_K$\@. If $(x_\alpha) \subset K$ is a net with limit $x$, then 
$\lim_\alpha w(x_\alpha)=w(x)$\@. Therefore $\lim_\alpha \rho(x_\alpha,x)=0$\@. It follows that
$\lim_\alpha f(x_\alpha)=f(x)$, so $f|_K$ is continuous at $x$\@.\end{proof}

Now the above result can be translated into a more general setting. 

\begin{prop}\label{faceloc}
Let $f\colon X \rightarrow {\mathbb R}$ be a convex lower semicontinuous function which is bounded on compact
subsets. Then for every compact convex subset $K \subset X$ and every open slice $S \subset K$\@,
there is a face $F \subset S$ of $K$ such that $f|_K$ is constant and continuous on $F$\@.
\end{prop}

\begin{proof} By Lemma~\ref{imbed}, $Z=\bigcup_{n=1}^\infty n\, \mbox{aconv}(K)$ is a dual Banach space and $f|_Z$ is weak* lower semicontinuous. Then we can apply the previous proposition.
\end{proof}

It is clear that the last two results are true for countably many functions simultaneously.

\begin{rema} We do not know if the function $f$ in Proposition \ref{facedual} and \ref{faceloc} can be assumed to be defined only on $K$\@. Notice that if $\norm$ is a strictly convex norm on $Z$ then $f(x)=-\sqrt{1-||x||^2}$ is a strictly convex weak* lower semicontinuous function on $(B_Z,w^*)$ that cannot be extended to a convex function on $Z$\@. 
\end{rema}

\section{Exposed points}

Notice that if a strictly convex function is constant on a face of a compact $K$, then necessarily that face should be an exposed point of $K$\@. Having this in mind, Propositions~\ref{facedual} and~\ref{faceloc} can be rewritten. As in the previous section $Z=W^*$ is a dual Banach space endowed with the weak$^*$ topology and we understood all the topological notions referred to that topology.

\begin{prop}\label{expo}
Let $f\colon Z \rightarrow {\mathbb R}$ be a strictly convex lower semicontinuous function which is bounded on compact subsets. If $K\subset Z$ is compact convex, then there exists a $G_{\delta}$ dense set of elements of $W$ exposing points of $K$ at which $f|_K$ is continuous.
\end{prop}
\begin{proof} It follows straightforward from Proposition~\ref{facedual}.
\end{proof}

In particular, we retrieve the following result, which is usually proved in the frame of G\^{a}teaux Differentiability Spaces~\cite[Corollary 2.39 and Theorem 6.2]{phelps}.

\begin{coro}[Asplund, Larman\textendash{}Phelps]
Let $Z$ be a strictly convex dual Banach space. Then every convex compact is the closed convex hull of its exposed points.
\end{coro}

\begin{proof}[Proof of Theorem~\ref{main2}]
 It follows straightforward from Proposition~\ref{faceloc}.
\end{proof}

\begin{coro}
Assume that $K \in \mathcal{SC}(X)$\@. Then $K$ is the closed convex hull of its exposed points.
\end{coro}
\begin{proof} Thanks to Theorem~\ref{main1} it can be reduced to the previous corollary.
\end{proof}

Notice that the previous result is far from being a characterization. For instance, consider $X=C([0,\omega_1])^*$ and $K=(B_X,w^*)$. Then $X$ has the Radon\textendash{}Nikod\'ym Property and thus there exist \emph{strongly exposed points} of $K$ \cite[Theorem 3.5.4]{bourgin}. Nevertheless, Talagrand's argument in~\cite[Theorem 5.2.(ii)]{DGZ} shows that $K\notin \mathcal{SC}(X,w^*)$. Indeed, the result of Larman and Phelps mentioned aboved states that Banach spaces for which each weak* compact convex subset has an exposed point are exactly dual spaces of a G\^{a}teaux Differentiability Space.
 
\begin{rema} A point $x$ in a subset $C$ of a normed space $(Z,\norm)$ is said to be a \emph{farthest point} in $C$ if there exists $y\in Z$ such that $\|y-x\|\geq \sup\{\|y-c\|:c\in C\}$\@. If $\norm$ is strictly convex then every farthest point of $C$ is exposed by a functional in $Z^\ast$\@. In addition, it was shown in~\cite{DZ} that there exists a weak* compact subset of $\ell_1$ that has no farthest points, so the existence of exposed points does not imply the existence of farthest points. On the other hand, suppose that $Z$ is a strictly convex dual Banach space, $C$ is a compact subset of $Z$ and $x$ is a farthest point in $C$ with respect to $y\in Z$\@. Consider the symmetric $\rho(u,v)=\frac{\|u-y\|^2+\|v-y\|^2}{2} - \|\frac{u+v}{2}-y\|^2$\@. Then $x$ is a $\rho-$denting point of $C$, that is, admits slices with arbitrarily small $\rho$-diameter. Indeed, if $\delta = \frac{ \epsilon}{1+2\|x-y\|+2\|y\|}$ then every slice of $C$ that does not meet $B(y, \|y-x\|-\delta)$ has $\rho$-diameter less than $\epsilon$\@. 
\end{rema}

Typically a variational principle provides strong minimum for certain functions after a small perturbation. But in the compact setting, a lower semicontinuous function already attains its minimum. Nevertheless, inspired by Stegall's variational principle~\cite[Theorem 11.6]{banach}, we have obtained the following result.

\begin{prop}
Suppose that $K \in \mathcal{SC}(X)$ and let $f\colon K \rightarrow {\mathbb R}$ be a lower semicontinuous function. Given $\epsilon>0$\@, there exists an affine continuous function $w$ on $K$ with oscillation less than $\epsilon$ such that $f+w$ attains its minimum exactly at one point. Moreover, if $X$ is a dual Banach space then $w$ can be taken from the predual with norm less than $\epsilon$\@.
\end{prop}
\begin{proof} 
 By the embedding it is enough to consider the Banach case.
 Let $m$ be the minimum of $f$ and take $M>0$ such that $K\subset M B_X$\@. Consider the compact set
 \[ H = \left\{ (x,t): f(x) \leq t \leq m+\epsilon M \right\} \] 
 and take its convex closed envelop $A$\@. By Proposition~\ref{property}, 
 $A \in \mathcal{SC}(X \times {\mathbb R})$\@. The functional on $X \times {\mathbb R}$
 given by $(0,1)$ attains its minimum on $A$\@. Proposition~\ref{expo} provides a small perturbation of the form $(w,1)$, with $\|w\|<\epsilon$, attaining its minimum on $A$ at one single point $(x_0, t_0)$\@. Notice that $t_0=f(x_0)$ and $f(x_0)+w(x_0)\leq m+\epsilon M$\@. If $y \in K$, then either $f(y) \leq m+\epsilon M$ and $(y,f(y)) \in A$, or $f(y)>m+\epsilon M\geq f(x_0)+w(x_0)$\@.
\end{proof}

\section{Ordinal indices}

Let $K$ be a convex and compact subset of a locally convex topological vector space and $\rho$ a symmetric on $K$\@. We consider the following set derivations:
\begin{align*}
[K]'_\epsilon &= \left\{x\in K : x\in S \text{ slice of } K \Rightarrow \rho\mbox{-diam}(S)\geq \epsilon\right\};\\
\<K\>'_\epsilon &= \left\{x\in K : x\in U \text{ open }  \Rightarrow \rho\mbox{-diam}(S)\geq \epsilon\right\}\,.
\end{align*}
The iterated derived sets are defined as $[K]^{\alpha+1}_\epsilon = [[K]^{\alpha}_\epsilon]'_\epsilon$, $\<K\>^{\alpha+1}_\epsilon = 
\<\<K\>^{\alpha}_\epsilon\>'_\epsilon$ and intersection in case of limit ordinals. If there exists some ordinal such that $[K]^{\alpha}_\epsilon =\emptyset$, then we set $Dz_\rho(K,\epsilon)=\min\left\{\alpha: {[K]^\alpha_\epsilon} = \emptyset\right\}$\@. Otherwise, we take $Dz_\rho(K,\epsilon)=\infty$, which is beyond the ordinals. The $\rho$-dentability index of $K$ is defined by $Dz_{\rho}(K) = \sup_{\epsilon>0} {Dz(K,\epsilon)}$\@. The $\rho$-Szlenk index of $K$, $Sz_{\rho}(K)$, is defined the same way. Obviously $Sz_{\rho}(K)\leq Dz_{\rho}(K)$\@. Set derivations with respect to a symmetric were introduced in \cite{FOR} in order to characterize dual Banach spaces admitting a dual strictly convex norm. 

\begin{prop}\label{sym} Let $K$ be a convex compact subset of a locally convex space. Then the following assertions are equivalent:
\begin{itemize}
\item[a)] $K\in \mathcal{SC}$\@;
\item[b)] there exists a symmetric $\rho$ on $K$ such that $Dz_{\rho}(K)\leq \omega$\@;
\item[c)] there exists a symmetric $\rho$ on $K$ such that $Dz_{\rho}(K)\leq \omega_1$\@.
\end{itemize}
\end{prop}
\begin{proof} Let $f$ be a bounded function witnessing that $K\in \mathcal{SC}$ and assume that $f$ takes values in $[0,1]$\@. For a fixed $\epsilon>0$, take $N>1/\epsilon$ and define the closed convex subsets $F_n = \left\{x\in K: f(x)\leq 1-n/N\right\}$ for $n=0,\ldots N$\@. Take 
\[ \rho(x,y)=\frac{f(x)+f(y)}{2}-f\bigl(\frac{x+y}{2}\bigr)\,.\]
We claim that $[K]'_\epsilon \subset F_1$\@. Let $x_0\in K\smallsetminus F_1$\@. By the Hahn\textendash{}Banach theorem, there exists a slice $S$ of $K$ such that $x_0\in S$ and $S\cap F_1 = \emptyset$\@. If $x,y\in S$, then $\frac{x+y}{2}\in S$ and $\rho(x,y)\leq 1 - (1-1/N)=1/N$\@. Thus, $\rho\mbox{-diam}(S)<\epsilon$ and $x_0\notin {[K]'_\epsilon}$\@. By iteration, we get that $[K]^{N}_\epsilon \subset F_N$ and hence $[K]^{N+1}_\epsilon = \emptyset$\@. Therefore, $ Dz_\rho(K,\epsilon)<\omega$ for each $\epsilon>0$\@. 

Now suppose that $Dz_\rho(K)\leq \omega_1$\@. Notice that indeed $Dz_\rho(K)< \omega_1$\@. By Corollary~\ref{SCstar}, it suffices to show that $K$ has $(\ast)$ with slices. For each $n\in \mathbb{N}$ and $\alpha<Dz_\rho(K, 1/n)$ consider the family
\[ \mathcal{U}_{n,\alpha} = \left\{ S : S \text{ slice of } K, [K]_{1/n}^{\alpha+1}\cap S=\emptyset, \rho\mbox{-diam}([K]_{1/n}^\alpha\cap S)<1/n\right\}\,.\]
Given distinct $x,y\in K$, take $n$ so that $\rho(x,y)>1/n$ and let $\alpha$ be the least ordinal such that $\{x,y\}\cap[K]_{1/n}^{\alpha+1}$ is at most a singleton. Then it is clear that there is a slice in $\mathcal{U}_{n,\alpha}$ containing either $x$ or $y$, and no slice in $\mathcal{U}_{n,\alpha}$ contains both points.
\end{proof}

\begin{rema} By using deep results of descriptive set theory, Lancien proved in \cite{lancien} that there exists an universal function $\psi\colon [0,\omega_1)\rightarrow[0,\omega_1)$ such that
$Dz_{\norm}(B_{X^*})\leq \psi(Sz_{\norm}(B_{X^*}))$ whenever $X$ is a Banach space such that $Sz_{\norm}(B_{X^*})<\omega_1$\@. We do not know if a similar statement holds when the norm is replaced by a symmetric. 
\end{rema}

We shall show that we cannot change symmetric by metric in Proposition \ref{sym}.  That would imply that $K$ is a Gruenhage compact, which is a strictly stronger condition that being in $\mathcal{SC}$~\cite[Theorem 2.4]{smith}. By~\cite[Lemma 7.1 and Proposition 7.4]{stegall}, a compact space $K$ is Gruenhage if and only if there exist a countable set $D$, a family of closed sets $\{A_d: d\in D\}$ and families $(\mathcal{U}_{d})_{d\in D}$ of open sets such that the family $\{A_d\cap U : U\in \mathcal{U}_{d}\}$ is pairwise disjoint for each $d\in D$ and the family $\{A_d\cap U : U\in \mathcal{U}_{d}\}$ separates the points of $K$\@.

\begin{prop} Let $K$ be a compact space. Then the following assertions are equivalent:
\begin{itemize}
\item[a)] $K$ is Gruenhage;
\item[b)] there exists a metric $d$ on $K$ such that $Sz_{d}(K)\leq \omega$\@;
\item[c)] there exists a metric $d$ on $K$ such that $Sz_{d}(K)\leq \omega_1$\@.
\end{itemize}

\end{prop}
\begin{proof} 
If $K$ is a Gruenhage compact space, then the same construction used in the proof of \cite[Theorem 2.8]{raja2} provides a metric on $K$ such that $Sz_{d}(K)\leq \omega$\@.
 
Now assume that $d$ is a metric on $K$ with countable Szlenk index. Let $\mathcal{B} = \bigcup_{m\in\N} \mathcal{B}_m$ be a basis of the metric topology such that every $\mathcal{B}_m$ is discrete. Consider the open sets $U_V^{n,\alpha} = \bigcup \left\{U : U \text{ open}, \<K\>^\alpha_{2^{-n}} \cap U\subset V\right\}$ and the families $\mathcal{U}_m^{n,\alpha} = \{U_V^{n,\alpha} : V\in \mathcal{B}_m\}$\@. Then $\{\<K\>^\alpha_{2^{-n}}\cap U : U\in\mathcal{U}_m^{n,\alpha}\}$ is pairwise disjoint for each $n,m\in \N$ and $\alpha<Dz(K,2^{-n})$\@. Given distinct $x,y \in K$ take $V\in\mathcal{B}_m$ such that $x\in V$ and $y\notin V$\@. Fix $n$ such that $B_d(x,2^{-n+1})\subset V$\@. Let $\alpha$ be the least ordinal so that $x\notin \<K\>^{\alpha+1}_{2^{-n}}$\@. Then there is an open subset $U$ of $K$ such that $x\in \<K\>_{2^{-n}}^\alpha\cap U$ and $\mbox{diam}(\<K\>_{2^{-n}}^\alpha\cap U)\leq 2^{-n}$\@. Thus $x\in \<K\>_{2^{-n}}^\alpha\cap U_V^{n,\alpha}\subset V$, so $y\notin\<K\>_{2^{-n}}^\alpha\cap U_V^{n,\alpha}$\@.    
\end{proof}


{\footnotesize

\begin{thebibliography}{99}
\bibitem{bourgin}{\sc R.D. Bourgin}:
{\it ~Geometric Aspects of Convex Sets with Radon\textendash{}Nikod\'ym Property},
Lect. Notes in Math. 993, Springer Verlag, 1980.

\bibitem{DGZ}{\sc R. Deville, G. Godefroy, V. Zizler}:
{\it Smoothness and renormings in Banach spaces}, Pitman Monographs and Surveys in Pure and Applied Mathematics, 64. Longman Scientific \& Technical, Harlow, 1993.

\bibitem{DZ}{\sc R. Deville, V.E. Zizler}:
{~Farthest points in {$w^*$}-compact sets},
{\it Bull. Austral. Math. Soc.} {\bf 38} (1988), no. 3, 433--439.

\bibitem{banach} {\sc M. Fabian, P. Habala, P. H\'ajek, V. Montesinos, J. Pelant, V. Zizler}:
{\it ~Functional analysis and infinite-dimensional geometry},
CMS Books in Mathematics/Ouvrages de Math\'ematiques de la SMC, 8.
Springer-Verlag, New York, 2001.

\bibitem{GL}{\sc G. Godefroy and D. Li}:
{~Strictly convex functions on compact convex sets and applications},
{Functional analysis}, 182--192, Narosa, New Delhi, 1998.

\bibitem{herve}{\sc M. Herv\'e}:
{~Sur les repr\'esentations int\'egrales \`a l'aide des points extr\'emaux dans un ensemble compact convexe m\'etrisable},
{\it C. R. Acad. Sci. Paris} {\bf 253} (1961) 366--368.

\bibitem{FOR}{\sc S. Ferrari, J. Orihuela, M. Raja}:
{~Weak metrizability of spheres}, to apear.

\bibitem{lancien}{\sc G. Lancien}:
{~On the Szlenk index and the weak*-dentability index},
{\it Quart. J. Math. Oxford} (2) {\bf 47} (1996), no. 185, 59--71.

\bibitem{MOTV}{\sc A. Molt\'o, J. Orihuela, S. Troyanski, M. Valdivia}:
{~On weakly locally uniformly rotund Banach spaces}, 
{\it J. Funct. Anal.} {\bf 163} (1999), no. 2, 252--271.

\bibitem{NG}{\sc K. F. Ng}: 
{~On a theorem of Dixmier},
{\it Math. Scand.} {\bf 29} (1971), 279-280 (1972). 

\bibitem{OST}{\sc J. Orihuela, R. Smith, S. Troyanski}:
{~Strictly convex norms and topology},
{\it Proc. London Math. Soc.} (3) {\bf 104} (2012) 197--222.

\bibitem{OR}{\sc L. Oncina, M. Raja}:
{~Descriptive compact spaces and renorming},
{\it Studia Math.} {\bf 165} (2004), no. 1, 39--52.

\bibitem{phelps}{\sc R. R. Phelps}:
{\it ~Convex Functions, Monotone Operators and Differentiability},
Lect. Notes in Math. 1364, Springer Verlag, 1993.

\bibitem{raja}{\sc M. Raja}:
{~Continuity at the extreme points},
{\it J. Math. Anal. Appl.} {\bf 350} (2009), no. 2, 436--438.

\bibitem{raja2}{\sc M. Raja}:
{~Compact spaces of Szlenk index $\omega$}, 
{\it J. Math. Anal. Appl.} {\bf 391} (2012), no. 2, 496--509. 

\bibitem{riba}{\sc N. K. Ribarska}:
{~Internal characterization of fragmentable spaces}
{\it Mathematika} {\bf 34} (1987), no. 2, 243--257.

\bibitem{riba2}{\sc N. K. Ribarska}:
{~A note on fragmentability of some topological spaces},
{\it C. R. Acad. Bulg. Sci.} {\bf 43} (1990), no.7, 13--15.

\bibitem{smith}{\sc R. Smith}:
{~Strictly convex norms, $G_\delta$-diagonals and non-Gruenhage spaces}, 
{\it Proc. Amer. Math. Soc.} {\bf 140} (2012), no. 9, 3117--3125.

\bibitem{stegall}{\sc C. Stegall}:
{~The topology of certain spaces of measures},
{\it Topology Appl.} {\bf 41} (1991), no. 1-2, 73--112.

\end{thebibliography}



}
\end{document}